\documentclass[11pt,english]{article}
\usepackage[]{fontenc}
\usepackage[latin9]{inputenc}
\usepackage{geometry}
\usepackage{amsthm}
\usepackage{amsmath}
\usepackage{amssymb}
\usepackage{setspace}
\usepackage{esint}
\makeatletter
\theoremstyle{plain}
\newtheorem{thm}{\protect\theoremname}[section]
  \theoremstyle{definition}
  \newtheorem{defn}[thm]{\protect\definitionname}
  \theoremstyle{plain}
  \newtheorem{prop}[thm]{\protect\propositionname}
  \theoremstyle{plain}
  \newtheorem{lem}[thm]{\protect\lemmaname}
  \theoremstyle{plain}
  \newtheorem{cor}[thm]{\protect\corollaryname}

\usepackage{lipsum}

\DeclareMathOperator*{\dls}{\bar{d}-limsup}
\DeclareMathOperator*{\dlim}{\bar{d}-lim}

\renewcommand\footnotemark{}

\makeatother

\usepackage{babel}
  \providecommand{\corollaryname}{Corollary}
  \providecommand{\definitionname}{Definition}
  \providecommand{\lemmaname}{Lemma}
  \providecommand{\propositionname}{Proposition}
\providecommand{\theoremname}{Theorem}

\begin{document}

\title{On the ratio ergodic theorem for group actions }

\author{Michael Hochman\thanks{Supported by ISF grant 1409/11}}
\maketitle
\begin{abstract}
We study the ratio ergodic theorem (RET) of Hopf for group actions.
Under a certain technical condition, if a sequence of sets $\{F_{n}\}$
in a group satisfy the RET, then there is a finite set $E$ such that
$\{EF_{n}\}$ satisfies the Besicovitch covering property. Consequently
for the abelian group $G=\oplus_{n=1}^{\infty}\mathbb{Z}$ there is
no sequence $F_{n}\subseteq G$ along which the RET holds, and in
many finitely generated groups, including the discrete Heisenberg
group and the free group on $\geq2$ generators, there is no (sub)sequence
of balls, in the standard generators, along which the RET holds. 

On the other hand, in groups with polynomial growth (including the
Heisenberg group, to which our negative results apply) there always
exists a sequence of balls along which the RET holds if convergence
is understood as a.e. convergence in density (i.e. omitting a sequence
of density zero). 
\end{abstract}

\section{\label{sec:Introduction}Introduction}

Let $G$ be a countable group acting from the left by measure-preserving
transformations on a measure space $(X,\mathcal{B},\mu)$, with the
action of $g\in G$ on $x\in X$ written $x\mapsto T^{g}x$. We assume
the action is ergodic. For a finite set $F\subseteq G$ and $\varphi:X\rightarrow\mathbb{R}$
let 
\[
S_{F}(\varphi)=\sum_{g\in F}\varphi\circ T^{g}
\]
The asymptotic behavior of $S_{F_{n}}(\varphi)$ as $F_{n}$ exhausts
the group, in some sense, is the subject of the ergodic theorem. In
this paper we are interested in the situation when $\mu$ is an infinite
(without loss of generality $\sigma$-finite) measure, in which case
the appropriate quantity to consider are the ratios 
\[
R_{F}(\varphi,\psi)=\frac{S_{F}(\varphi)}{S_{F}(\psi)}
\]
One says that $G$ satisfies a ratio ergodic theorem along $\{F_{n}\}$
if for every ergodic measure-preserving action of $G$ on a non-atomic
measure space, and every $\varphi,\psi\in L^{1}(\mu)$ with $\int\psi d\mu\neq0$,
we have 
\[
R_{n}(\varphi,\psi)\rightarrow\frac{\int\varphi d\mu}{\int\psi d\mu}\qquad\mu\mbox{-a.e.}
\]

For $G=\mathbb{Z}$ and $F_{n}=[1,n]\cap\mathbb{Z}$ the ratio ergodic
theorem was proved by Hopf in 1937, only a few years after the ergodic
theorems of von Neumann and Birkhoff. Unlike the latter theorems,
however, which have been extended to very general classes of groups
(see \cite{Nevo2006} for a recent survey), extensions of Hopf's theorem
have been slow to appear. Part of the reason is that, for a time,
it was believed that no extension is possible, due to an example of
Brunel and Krengel \cite{Krengel1985}, who showed for $\mathbb{Z}^{d}$,
$d\geq2$, that ratio ergodic theorem fails along $F_{n}=[0,n]^{d}$.
Nevertheless there is a ratio ergodic theorem for actions of $\mathbb{Z}^{d}$
along symmetric cubes $F_{n}=[-n,n]^{d}\cap\mathbb{Z}^{d}$. This
was first proved by Feldman under additional assumptions on the dynamics
of the action \cite{Feldman2007}, and we proved the general case
in \cite{Hochman2009e} (also for some more general sequences $F_{n}$).%
\footnote{Bowen and Nevo have also recently obtained a variant of the ratio
ergodic theorem for free groups, but with some additional randomization
which makes the problem somewhat different.%
}

The prospects for groups other than $\mathbb{Z}^{d}$ has remained
unclear. A pertinent fact from \cite{Hochman2009e} is that a certain
maximal inequality, which is central to most existing proofs, is actually
\emph{equivalent}, in the present context, to $\{F_{n}\}$ satisfying
the (right) Besicovitch covering property (Definition \ref{def:Besicovitch}
below). Sufficiency of this property was observed earlier by Becker
\cite{Becker1983} (see also \cite{Rudolph07,Lindenstrauss2006}).
The Besicovitch property is quite rare, and its failure puts into
question the validity of the ratio ergodic theorem for many groups.
But, on the other hand, as far as we know the maximal inequality and
the ratio ergodic theorem are not equivalent. Indeed for $\mathbb{Z}^{d}$
it took several decades to prove the latter once the former became
available. 

In this paper we present two main results. The first shows that, indeed,
the ratio ergodic theorem is quite rare, and is closely linked to
the Besicovitch property. The second, on the other hand, shows that
a certain weakening of it does hold more generally, including in groups
where the strong version above fails.

We begin with main negative result, which requires the following definition.
Let us say that $\{F_{n}\}$ is almost central if for every $g\in G$
there is a finite set $E$ with $F_{n}g\subseteq EF_{n}$ for all
$n$ (here and throughout we write $AB=\{ab\,:\, a\in A\,,\, b\in B\}$,
etc.). This holds trivially in abelian groups, and also for balls
in any finitely generated group.
\begin{thm}
\label{thm:main}If $\{F_{n}\}$ is almost central and satisfies the
ratio ergodic theorem, then there is a finite set $E$ such that $\{EF_{n}\}$
is Besicovitch.
\end{thm}
For a marginally stronger statement see Theorem \ref{thm:B-is-necessary}.
It seems possible that the Besicovitch property is necessary in general
for a ratio ergodic theorem, but this remains open.

We give two main applications. First,
\begin{thm}
\label{thm:Z-infinity}Let $\mathbb{Z}^{\infty}=\oplus_{n=1}^{\infty}\mathbb{Z}$.
Then the ratio ergodic theorem fails along every sequence $F_{n}\subseteq\mathbb{Z}^{\infty}$.
\end{thm}
The group $\mathbb{Z}^{\infty}$ is extremely {}``nice'' -- it is
abelian, amenable and residually finite. As such it is quite surprising
that the ratio ergodic theorem should fail. However one might object
that it is {}``infinite-dimensional''. So suppose now that $G$
is generated by a finite symmetric set $A$ and let $B_{n}=A^{n}$
be the associated {}``balls''. 
\begin{thm}
\label{thm:subseq-balls}Let $G$ be finitely generated and $\{B_{n}\}$
the sequence of balls with respect to some finite generating set.
Suppose that no sequence of balls has the Besicovitch property. Then
the ratio ergodic theorem fails along every sub-sequence $F_{i}=B_{n(i)}$.
\end{thm}
In particular, in the discrete Heisenberg group with the standard
generator set there is no sequence of balls which satisfies the ratio
ergodic theorem (see Section \ref{sub:Derivation-of-main-results}).
The same holds for free groups on $\geq2$ generators. In these examples
we cannot yet rule out the existence of some other sequence along
which it does hold, and one should note that there are finitely generated
groups, such as the lamplighter groups, for which balls are not the
right averaging sets to consider. But for groups of polynomial growth
all known ergodic theorems do hold along balls, and it would be quite
surprising if some other good sequence exists. 

We turn now to our positive result which shows that, if one accepts
a slightly weaker notion of convergence, then there is a version of
the ratio ergodic theorem which holds in greater generality. Recall
that a group $G$ has polynomial growth if the balls $B_{n}$ satisfy
$|B_{n}|\leq c'n^{c}$ for constants $c,c'$. Define the (upper) density
of a set $I\subseteq\mathbb{N}$ of integers by
\[
\overline{d}(I)=\limsup_{N\rightarrow\infty}\frac{1}{N}|I\cap[1,N]|
\]
A sequence $a_{n}$ converges to $a$ in density if 
\[
\overline{d}(n\,:\,|a-a_{n}|>\varepsilon)=0\qquad\mbox{for every }\varepsilon>0
\]
We denote this limit by $a_{n}\xrightarrow{\overline{d}}a$ or $\dlim a_{n}=a$.
This operator satisfies all the usual properties of limits. 
\begin{thm}
\label{thm:polynomial-density-thm}Let $G$ be a group of polynomial
growth and $B_{n}$ as above. Then there is a subsequence $F_{i}=B_{n(i)}$
along which the ratio ergodic theorem holds in density, i.e. 
\begin{equation}
R_{F_{i}}(\varphi,\psi)\xrightarrow{\overline{d}}\frac{\int\varphi d\mu}{\int\psi d\mu}\qquad\mu\mbox{-a.e.}\label{eq:2}
\end{equation}
for any ergodic measure-preserving action of $G$ and any $\varphi,\psi\in L^{1}(\mu)$
with $\int\psi d\mu\neq0$.
\end{thm}
The proof is given in Section \ref{sec:groups-of-polynomial-growth}.
Thus there are cases, such as the discrete Heisenberg group, where
no sequence of balls satisfied the ratio ergodic theorem, but there
exist sequences along for which the density version is valid. Our
arguments are special to groups of polynomial growth but some parts
carry over also to groups of sub-exponential growth (groups with $|B_{n}|=o(c^{n})$
for all $c>1$). We do not know if and when similar modifications
of the ratio ergodic theorem hold in more general groups.

There are other possible weakenings of the ratio ergodic theorem.
One possibility~is to require a.e. pointwise convergence of the ratios
to other limit functions. This phenomenon has been recently observed
in certain algebraic settings involving {}``large'' groups acting
on infinite measure spaces, see e.g. the introduction of \cite{BowenNevo2012}.%
\footnote{I am grateful to Amos Nevo for drawing my attention to this phenomenon.%
} However, our negative results exclude this as well; in the proofs
we construct actions for which the ratios diverge. 

The rest of the paper is divided into two sections: In Section \ref{sec:Besicovitch-is-necessary}
we develop the necessary combinatorics and prove Theorems \ref{thm:subseq-balls},
\ref{thm:Z-infinity} and \ref{thm:main}. We prove Theorem \ref{thm:polynomial-density-thm}
in Section \ref{sec:groups-of-polynomial-growth}.

\section{\label{sec:Besicovitch-is-necessary}Besicovitch is necessary}

\subsection{\label{sub:Combinatorial-preliminaries.}Combinatorial preliminaries}

In this section $\{F_{n}\}$ denotes a sequence of finite subsets
of $G$, all containing the identity element $1_{G}$. We begin with
some combinatorial definitions. 

A collection $\{E_{i}\}_{i\in I}$ of subset of $G$ is said to have
multiplicity $k$ at a point $g$ if $g$ belongs to $k$ of the sets.
The multiplicity of $\{E_{i}\}$ is the smallest $k$ such that all
points have multiplicity $\leq k$. The\,following definition is
classical in analysis where, instead of translates of sets in a group,
one considers balls in a metric space. 
\begin{defn}
\label{def:Besicovitch}$\{F_{n}\}$ satisfies the Besicovitch covering
property (or, more concisely, $\{F_{n}\}$ is Besicovitch) if there
is a constant $C$ such that, for every finite $A\subseteq G$ and
any family of sets of the form $\{F_{n(g)}g\}_{g\in A}$ there is
a subset $A'\subseteq A$ such that $\{F_{n(g)}g\}_{g\in A'}$ covers
$A$ and has multiplicity $\leq C$; equivalently, 
\[
1_{A}\leq\sum_{g\in A'}1_{F_{n(g)}g}\leq C
\]

\end{defn}
It is easy to see that any finite sequence $\{F_{n}\}_{n=1}^{N}$
is Besicovitch and that $\{F_{n}\}_{n=1}^{\infty}$ is Besicovitch
if and only if $\{F_{n}\}_{n=n_{0}}^{\infty}$ is Besicovitch for
every $n_{0}$. In this section we rely primarily on the following
characterization of the Besicovitch property. Define a (right) incremental
sequence to be a finite sequence $(F_{n(i)}g_{i})_{i=1}^{k}$ such
that $g_{j}\notin\bigcup_{i<j}F_{n(i)}g_{i}$ and $n(1)\geq n(2)\geq\ldots\geq n(k)$.
\begin{prop}
\label{prop:high-multiplicity-incremental-sequences}If $\{F_{n}\}$
is not Besicovitch then for every $k$ there is an incremental sequence
of multiplicity $k$ (equivalently, with $1_{G}$ belonging $k$ members
of the sequence). If in addition $F_{n}$ are symmetric and increasing,
the converse holds. \end{prop}
\begin{proof}
We include the standard proof for completeness. If $\{F_{n}\}$ is
not Besicovitch then, given $k$, there is a family $\{F_{n(i)}g_{i}\}_{i=1}^{\ell}$
such that any sub-collection covering all the $g_{i}$ is of multiplicity
$k$. We may assume $n(i)$ are non-increasing. Choose an incremental
subsequence $\{F_{n(i_{m})}g_{i_{m}}\}$ greedily: let $i_{1}=1$,
and if $i_{1},\ldots,i_{m}$ are defined take $i_{m+1}$ to be the
minimal $i>i_{m}$ satisfying $g_{i}\notin\bigcup_{j\leq m}F_{n(i_{j})}g_{i_{j}}$.
The resulting sequence covers all the $g_{i}$ (here we use $1_{G}\in F_{n}$),
thus has multiplicity $k$. We can assume the multiplicity is realized
at $1_{G}$ by applying an appropriate right translation to the sets.

In the other direction, given $k$ let $\{F_{n(i)}g_{i}\}$ be an
incremental sequence of multiplicity $k$. Thus $g_{i}\notin F_{n(j)}g_{j}$
for all $i>j$; by symmetry $g_{j}\notin F_{n(j)}^{-1}g_{i}=F_{n(j)}g_{i}$.
Since $n(j)\geq n(i)$ this shows that $g_{j}\notin F_{n(i)}g_{i}$
also for $i<j$. Thus this holds for all $i\neq j$, and the only
sub-collection of $\{F_{n(i)}g_{i}\}$ that covers all the $g_{i}$
is the full sequence, whose multiplicity is $k$. Since $k$ was arbitrary,
$\{F_{n}\}$ is not Besicovitch. 
\end{proof}
Our main interest is in sequences for which the Besicovitch property
fails. We require a slightly stronger property: 
\begin{defn}
\label{def:strongly-non-B}$\{F_{n}\}$ is strongly non-Besicovitch
if for every finite $E\subseteq G$ there is a finite set $E\subseteq\widetilde{E}\subseteq G$
with $1_{G}\in\widetilde{E}$ such that $\{\widetilde{E}F_{n}\}$
is not Besicovitch.
\end{defn}
We note two situations where this property holds: first, when no sequence
$F'_{n}\supseteq F_{n}$ is Besicovitch (take $\widetilde{E}=E\cup\{1_{G}\}$).
Second, if $G$ is finitely generated, $B_{n}$ are balls, and no
sub-sequence of balls $\{B_{n(i)}\}_{i=1}^{\infty}$ is Besicovitch
then every sub-sequence is strongly non-Besicovitch; indeed given
a finite set $E$ take $\widetilde{E}=B_{m}$, so that $\widetilde{E}B_{n}=B_{n+m}$.

The sets we consider later will also satisfy the following property,
which was already mentioned in the introduction:
\begin{defn}
\label{def:almost-central}$\{F_{n}\}$ is almost central if for every
$g\in G$ there is a finite set $E\subseteq G$ with $F_{n}g\subseteq EF_{n}$
for all $n$.
\end{defn}
The two primary examples are when $G$ is abelian, in which case we
can take $E=\{g\}$; and when $F_{n}=B_{k(n)}$ are balls, since then
if $g\in B_{m}$ then $F_{n}g\subseteq B_{m}F_{n}$. The latter example
shows that this property does not actually have any connection to
abelianness of the group, since it holds for balls in any finitely
generated group. It is clear that the definition is equivalent to
the statement that for every finite $E\subseteq G$ there is a finite
$E'\subseteq G$ with $F_{n}E\subseteq E'F_{n}$ for all $n$.

We now derive some properties of strongly non-Besicovitch and almost-central
sequences. 
\begin{lem}
\label{lem:strongly-non-B-consequences}Suppose that $\{F_{n}\}$
is strongly non-Besicovitch. For every finite $E\subseteq G$ there
is a finite $\widetilde{E}\subseteq G$ containing $E$, and such
that for every finite $H\subseteq G$ there are arbitrarily long incremental
sequences $\{F_{n(i)}g_{i}\}_{i=1}^{\ell}$ satisfying \end{lem}
\begin{enumerate}
\item [(a)] $Eg_{i}\cap F_{n(j)}g_{j}=\emptyset$ and $Eg_{i}\cap Eg_{j}=\emptyset$
for $i>j$. 
\item [(b)] $Eg_{i}\cap H=\emptyset$ for all $i$.
\item [(c)] $\widetilde{E}^{-1}\cap F_{n(i)}g_{i}\neq\emptyset$ for all
$i$. \end{enumerate}
\begin{proof}
Let $E$ be given and assume without loss of generality that $1_{G}\in E$.
Let $\widetilde{E}$ be associated to $E^{-1}E$ as in Definition
\ref{def:strongly-non-B}. Let $H\subseteq G$ be finite, let $k=\ell+|E^{-1}H|$,
and, using Proposition \ref{prop:high-multiplicity-incremental-sequences},
choose an incremental sequence $\{\widetilde{E}F_{n(j)}g_{j}\}_{j=1}^{\ell}$
for $\{\widetilde{E}F_{n}\}$ with $1_{G}\in\bigcap\widetilde{E}F_{n(j)}g_{j}$. 

For $j<i$ we have $g_{i}\notin\widetilde{E}F_{n(j)}g_{j}$. Since
$1_{G}\in\widetilde{E}$ this implies $g_{i}\notin F_{n(j)}g_{j}$,
hence $\{F_{n(j)}g_{j}\}_{j=1}^{\ell}$ is incremental. 

For (a), let $i>j$. Then $g_{i}\notin\widetilde{E}F_{n(j)}g_{j}$
and $E^{-1}\subseteq\widetilde{E}$ imply $Eg_{i}\cap F_{n(j)}g_{j}=\emptyset$;
similarly $g_{i}\notin\widetilde{E}F_{n(j)}g_{j}$ and $1_{G}\in F_{n(i)}$,
together with the definition of $\widetilde{E}$, give $Eg_{i}\cap Eg_{j}=\emptyset$. 

For (b), note that since $1_{G}\in F_{n}$, by the incremental property,
all the $g_{j}$ are distinct. $Eg_{j}\cap H\neq\emptyset$ implies
$g_{j}\in E^{-1}H$, so after removing at most $|E^{-1}H|$ elements
of the sequence we are left with an incremental sequence of length
$\ell$ which, in addition to the above, satisfies (b).

(c) follows from $1_{G}\in\widetilde{E}F_{n(i)}g_{i}$.\end{proof}
\begin{lem}
\label{cor:strongly-non-B-almost-central-consequences} Suppose that
$\{F_{n}\}$ is strongly non-Besicovitch and almost central. Then
for every finite $D,E\subseteq G$ there exists a finite $H\subseteq G$,
with $D\subseteq H$, satisfying the following property: for every
$\ell$ there is an incremental sequence $\{F_{n(i)}g_{i}\}_{i=1}^{\ell}$
such that
\begin{enumerate}
\item [(i)]$Eg_{i}\cap F_{n(j)}eg_{j}=\emptyset$ and $Eg_{i}\cap Eg_{j}=\emptyset$
for all $i\neq j$ and $e\in E$.
\item [(ii)]$Eg_{i}\cap H=\emptyset$ for all $i$.
\item [(iii)]$F_{n(i)}eg_{i}\cap H\neq\emptyset$ for every $i$ and $e\in E$.
\end{enumerate}
\end{lem}
\begin{proof}
Let $D,E$ be given, without loss of generality $1_{G}\in D\cap E$.
Using almost centrality let $E'$ be such that $E'F_{n}\supseteq F_{n}(E\cup E^{-1})$,
and assume $1_{G}\in E'$ (otherwise just add $1_{G}$ to it). Let
$E''=(E')^{-1}E$, let $\widetilde{E''}$ be as in the previous lemma,
and apply the previous lemma to $H=D(E')^{-1}(\widetilde{E''})^{-1}$.
We obtain arbitrarily long incremental sequences $\{F_{n(i)}g_{i}\}$
satisfying (a)--(c). Property (ii) is just (b).

For (i), we already know that $E''g_{i}\cap F_{n(j)}g_{j}=\emptyset$
for $i>j$. Using the definition of $E''$ this gives $Eg_{i}\cap F_{n(i)}Eg_{j}=\emptyset$
for $i>j$. This is the same as $F_{n(i)}^{-1}Eg_{i}\cap Eg_{j}=\emptyset$
for $i>j$, which, by symmetry of $F_{n(i)}$, is just $F_{n(i)}Eg_{i}\cap Eg_{j}=\emptyset$
for $i>j$. Thus this relation holds for all $i\neq j$. (i) follows
using $1_{G}\in E\cap F_{n(i)}$.

For (iii), by choice of $E'$ for every $e\in E$ we have $F_{n(i)}e^{-1}\subseteq E'F_{n(i)}$,
hence $F_{n(i)}\subseteq E'F_{n(i)}e$. Combined with (c) this implies
that $(\widetilde{E''})^{-1}\cap E'F_{n(i)}eg_{i}\neq\emptyset$,
hence $(E')^{-1}(\widetilde{E''})^{-1}\cap F_{n(i)}eg_{i}\neq\emptyset$.
(iii) follows since $H\supseteq(E')^{-1}\widetilde{(E''})^{-1}$.
\end{proof}

\subsection{\label{sub:Necessity}Necessity}

In this section we add the assumption that the sets $F_{n}$ are symmetric,
and continue to assume $1_{G}\in F_{n}$. It will be convenient to
write $S_{F}(\varphi,x)$ and $R_{F}(\varphi,\psi,x)$ instead of
$S_{F}(\varphi)(x)$, $R_{F}(\varphi,\psi)(x)$. 
\begin{thm}
\label{thm:B-is-necessary}If $\{F_{n}\}$ is strongly non-Besicovitch
and almost central then there is an ergodic measure-preserving action
of $G$ on a non-atomic measures space $(X,\mathcal{B},\mu)$, and
functions $\varphi,\psi\in L^{1}(\mu)$ with $\int\psi\neq0$, such
that $R_{F_{n}}(\varphi,\psi)$ diverges a.e. as $n\rightarrow\infty$.
\end{thm}
Theorem \ref{thm:main} is then a formal consequence of Theorem \ref{thm:B-is-necessary},
since if $\{F_{n}\}$ is not strongly non-Besicovitch then $\{EF_{n}\}$
is Besicovitch for some finite set $E$.

The construction that is the proof of Theorem \ref{thm:B-is-necessary}
proceeds by cutting and stacking. We give full details below, but
let us first give an informal overview for readers familiar with the
method. Suppose we have defined a large {}``stack'' whose shape
a finite set $E\subseteq G$, and a pair of real-valued functions
$\varphi,\psi>0$ with $\left\Vert \varphi\right\Vert _{1},\left\Vert \psi\right\Vert _{1}<\infty$,
corresponding to an $\mathbb{R}$-coloring of $G_{0}$. Applying the
corollary to $E$ and $D=E$ we obtain a set $H\supseteq E$, and,
fixing a large $N$, an incremental sequence $\{F_{m(i)}\gamma_{i}\}_{i=1}^{N}$
with the associated properties. Now, cut the original stack into $N$
copies of equal mass, and translate them to $E\gamma_{i}$, $i=1,\ldots,N$,
which by the corollary are pairwise disjoint and disjoint from $H$.
Add new mass to the sites corresponding to $H$ (which is empty) in
the new stack, and on it define $\varphi$ to take very large negative
value $v$, and define $\psi$ to be zero there. Also add new mass
where necessary in $\bigcup_{i=1}^{N}F_{m(i)}E\gamma_{i}$, defining
$\varphi,\psi$ to be $0$ there. If $v$ is negative enough in a
manner depending only on the original stack, this forces the ratios
over $F_{n(i)}e\gamma_{i}$ for $e\in E$ to be $\leq-1$; but the
total change to $\left\Vert \varphi\right\Vert $ is $|H|v/N$, which
can be made arbitrarily small by choosing $N$ large. Iterating this
procedure, we can cause the ratios at the points corresponding to
the original $E$ to fluctuate between $\geq1$ and $\leq-1$, and
in the limit we obtain the desired counterexample.

We now carry this plan out in more detail. First we describe the cutting-and-stacking
scheme in the group context. Fix in advance the measure space $(\mathbb{R},Lebesgue)$.
We will define a compatible sequence of partial actions $T_{n}$.
By this we mean that: (i) for every $g$ we define a sequence of maps
$T_{n}^{g}$, $n=1,2,\ldots$ with increasing domains $X_{n,g}\subseteq\mathbb{R}$
and which extend each other in the sense that $T_{n+1}^{g}|_{X_{n,g}}=T_{n}^{g}$;
(ii) for $x\in X$, if both the expressions $T_{n}^{h}(T_{n}^{g}x)$
and $T_{n}^{hg}x$ are well defined (that is, if $x\in X_{n,g}$,
$T_{n}^{g}x\in X_{n,h}$, and $x\in X_{n,hg}$), then they are equal;
and (iii) writing $X_{n}=\bigcup_{g\in G}X_{n,g}$, for every $x\in X_{n}$
and every $g\in G$ we have $x\in X_{m,g}$ for some $m\geq n$. It
is clear that this defines in the limit an action of $G$ on $X=\bigcup_{n}X_{n}$
given by $T^{g}x=\lim_{n\rightarrow\infty}T_{n}^{g}x$. At the same
time, we will define $\varphi_{n},\psi_{n}:X_{n}\rightarrow X$ in
a compatible way, giving functions $\varphi,\psi:X\rightarrow\mathbb{R}$
in the limit.

The formulation above is somewhat unwieldy and the construction itself
will take \,the following form. At each stage $n$ we will have defined
a finite set $G_{n}\subseteq G$ and to each $g\in G$ associated
an interval $I_{n,g}=[a_{n,g},b_{n,g})\subseteq\mathbb{R}$ of length
$r_{n}>0$, independent of $g$, and with $I_{n,g}\cap I_{n,h}=\emptyset$
for $g\neq h$. For $x\in I_{n,g}$ the map $T_{n}^{h}$ is defined
if $hg\in G_{n}$, in which case $T_{n}^{h}x\in I_{n,hg}$ is the
point $a_{n,hg,}+(a_{n,g}-x)$ that occupies the same position in
$I_{n,hg}$ as $x$ occupies in $I_{n,g}$. Thus $X_{n}=\bigcup_{g\in G_{n}}I_{n,g}$
and $X_{n,h}=\{x\in X_{n}\,:\, x\in I_{n,g}$ and $hg\in X_{n}\}$.
What we have said so far ensures that (ii) is satisfied. To ensure
properties (i) and (iii) we first describe the transition from stage
$n$ to $n+1$, which is by {}``cutting and translating''. Given
$G_{n},\{I_{n,g}\}_{g\in G_{n}}$ for some $n$, we first choose a
large $N$ and choose elements $\gamma_{1},\ldots,\gamma_{N}$ of
$G$ such that the sets $G_{n}\gamma_{i}$ are pairwise disjoint.
Now partition each interval $I_{n,g}$ into $N$ intervals of equal
length 
\[
r_{n+1}=r_{n}/N
\]
Ordering these intervals from left to right, set $I_{n+1,g\gamma_{i}}$
to be the $i$-th sub-interval. 

We have so far defined intervals for $g\in G'_{n+1}=\bigcup_{i=1}^{N}G_{i}\gamma_{i}$,
and one may verify that the compatibility condition (i) holds. To
ensure (iii), fix a sequence $\Gamma_{n}\subseteq G$ of finite subsets
increasing to $G$ with $1\in\Gamma_{1}$, and define $G_{n+1}$ to
be any finite set containing $\Gamma_{n}G'_{n+1}$; to the new points
$g\in G_{n+1}\setminus G'_{n+1}$ assign arbitrary pairwise disjoint
intervals $I_{n+1,g}\subseteq\mathbb{R}\setminus X_{n}$. 

We will define by induction $G_{n},\{I_{n,g}\}_{g\in G_{n}}$ as above
with associated partial action $T_{n}$, and bounded functions $\varphi,\psi:X_{n}\rightarrow\mathbb{R}$
with $\left\Vert \varphi\right\Vert _{1},\left\Vert \psi\right\Vert _{1}<2$.
Furthermore we will have bounded functions $i_{n}:X_{n-1}\rightarrow\mathbb{N}$
such that for every $x\in[0,1]$ the maps $T_{n}^{g}x$ are defined
for all $g\in F_{i_{n}(x)},$ and 
\begin{eqnarray*}
R_{i_{n}(x)}(\varphi,\psi,x) &  & \left\{ \begin{array}{cc}
\geq1 & n\mbox{ odd}\\
\leq-1 & n\mbox{ even}
\end{array}\right.
\end{eqnarray*}
where we define $R_{n}(\varphi,\psi)$ as before in terms of the partial
action $T_{n}$. We also will ensure that $i_{n}(x)\rightarrow\infty$
for $x\in_{k=1}^{\infty}\bigcup X_{k}$. Assuming all this, it is
clear that, for the action $T$ defined in the limit, for every $k$
the ratios $R_{n}(\varphi,\psi,x)$ diverges for $x\in X_{k}$, and
hence diverge everywhere on $X=\bigcup X_{n}$. One point we have
not touched on is ergodicity of the limit action, we will come back
to this below.

It remains to describe the construction. At the first step we set
$G_{1}=\{1_{G}\}$, $I_{1,1_{G}}=[0,1]$, so $X_{1}=[0,1]$; define
$\varphi,\psi$ and $i_{1}$ to be identically $1$. Then all the
requisite properties hold.

Now suppose for some $n$ we have defined $G_{n},\{I_{n,g}\}_{g\in G_{n}}$,
$r_{n}$, $\varphi,\psi$, and $i_{n}$ as above. For simplicity we
assume $n$ is even, the odd case being the same. Let $i_{n}^{*}=\sup_{x\in X_{n-1}}i_{n}(x)<\infty$
and 
\begin{eqnarray*}
\Phi_{n} & = & \sup_{x\in[0,1]}|S_{F_{i_{n}(x)}}(\varphi,x)|\\
\Psi_{n} & = & \sup_{x\in[0,1]}|S_{F_{i_{n}(x)}}(\psi,x)|
\end{eqnarray*}
and choose 
\[
v=\Phi_{n}+\Psi_{n}
\]
so that $(v-\Phi_{n})/\Psi_{n}=1$. 

Let $H$ be the set associated to $D=E=G_{n}$ in Lemma \ref{lem:strongly-non-B-consequences}.
Choose $N$ large enough that 
\[
|H|\cdot v\cdot r_{n}/N<2-\int_{X_{n}}\varphi
\]
Applying the lemma, choose elements $\gamma_{1},\ldots,\gamma_{N}$
and indices $k_{1},\ldots,k_{N}>i_{n}^{*}$ such that 
\begin{eqnarray}
H\cap F_{k_{j}}g\gamma_{j} & \neq & \emptyset\qquad\mbox{for all }j\mbox{ and }g\in G_{n}\label{eq:d}\\
H\cap G_{n}\gamma_{j} & = & \emptyset\qquad\mbox{for all }j\label{eq:a}\\
G_{n}\gamma_{j}\cap F_{k_{j'}}g\gamma_{j'} & = & \emptyset\qquad\mbox{for }j\neq j'\mbox{ and }g\in G_{n}\label{eq:b}\\
G_{n}\gamma_{j}\cap G_{n}\gamma_{j'} & = & \emptyset\qquad\mbox{for }j\neq j'\label{eq:c}
\end{eqnarray}
 Let 
\[
G_{n+1}=\Gamma_{n+1}\left((\bigcup_{j=1}^{N}F_{k_{j}}G_{n-1}\gamma_{i})\cup H\right)
\]
Assign intervals of length $r_{n+1}=r_{n}/N$ to the elements of $G_{n+1}$
as follows: First partition each $I_{n,g}$ into $N$ intervals of
length $r_{n+1}$ and for $g\in G_{n}$ assign to $h=g\gamma_{j}$
the $j$-th sub-interval of $I_{n,g}$, which we call $I_{n+1,h}$.
So far there are no conflicts by (\ref{eq:c}) and the assignment
consists of disjoint intervals. To the remaining elements $h\in G_{n+1}\setminus\bigcup_{j=1}^{N}G_{n}\gamma_{j}$
associate arbitrary pairwise disjoint intervals $I_{n+1,h}\subseteq\mathbb{R}\setminus X_{n}$,
ensuring that the entire family $\{I_{n+1,g}\}_{g\in G_{n+1}}$ is
pairwise disjoint. This can easily be done since $X_{n}\subseteq\mathbb{R}$
is bounded.

For $x\in X_{n}\cap I_{n,g\gamma_{j}}$ for some $g\in G_{n}$, define
$i_{n+1}(x)=k_{j}$. Again, this is well defined by (\ref{eq:c}).

On $\bigcup_{h\in H}I_{n+1,h}$ set $\psi\equiv0$ and $\varphi\equiv v$.
There are no conflicts with previous definitions because of (\ref{eq:a}).

On the remaining mass, define $\varphi\equiv\psi\equiv0$ for $h\in H$.
There are no conflicts by (\ref{eq:b}).

Finally, in order to verify that $R_{i_{n+1}(x)}(\varphi,\psi,x)\geq1$
for $x\in X_{n-1}$, note that, by (\ref{eq:b}), if $x\in I_{n,g\gamma_{j}}$
for $g\in G_{n}$ then 
\begin{eqnarray*}
S_{F_{i_{n+1}(x)}}(\varphi,x) & = & S_{F_{i_{n}(x)}}(\varphi,x)+v\cdot|H\cap F_{i_{n+1}(x)}g\gamma_{j}|\\
S_{F_{i_{n+1}(x)}}(\psi,x) & = & S_{F_{i_{n}(x)}}(\psi,x)
\end{eqnarray*}
hence, by choice of $v$ and (\ref{eq:d}), we have $R_{F_{i_{n+1}(x)}}(\varphi,\psi,x)\geq1$.

While this construction ensures that the ratios diverge on a positive
fraction of the mass of a positive fraction of the ergodic component
of the action, these ergodic components may, a-priori, be atomic,
whereas we require divergence of the ratios on a non-atomic space.
The easiest solution is to introduce an intermediate step between
the stages of the construction, during which we create a large stack
with disjoint but very randomly placed copies of the previous stacks.
It is standard to show that the resulting action is ergodic, and the
new intermediate steps do not interfere with the construction above.
We omit the details.

\subsection{\label{sub:Derivation-of-main-results}Balls in finitely generated
groups.}

Let $G$ be a finitely generated and $B_{n}$ balls with respect to
some symmetric generating set. These are symmetric and contain the
identity and, as noted in the introduction, $\{B_{n}\}$ are almost
central. As noted in Section \ref{sub:Combinatorial-preliminaries.},
if no infinite subsequence $\{B_{n(i)}\}_{i=1}^{\infty}$ is Besicovitch
then every such subsequence is strongly non-Besicovitch. This and
Theorem Theorem \ref{thm:subseq-balls}.

The application to the Heisenberg group mentioned after Theorem \ref{thm:subseq-balls}
follows from:
\begin{prop}
\label{prop:Heisenberg}Let 
\[
G=\left\{ \left(\begin{array}{ccc}
1 & k & m\\
0 & 1 & n\\
0 & 0 & 1
\end{array}\right)\,\left|\begin{array}{c}
\\
k,m,n\in\mathbb{Z}\\
\\
\end{array}\right.\right\} 
\]
denote the discrete Heisenberg group with generating set 
\[
\{a^{\pm},b^{\pm}\}=\left\{ \left(\begin{array}{ccc}
1 & \pm1 & 0\\
0 & 1 & 0\\
0 & 0 & 1
\end{array}\right),\left(\begin{array}{ccc}
1 & 0 & 0\\
0 & 1 & \pm1\\
0 & 0 & 1
\end{array}\right)\right\} 
\]
Let $\{B_{n}\}$ be the associated sequence of balls. Then every infinite
subsequence $\{B_{n(k)}\}$ is non-Besicovitch.\end{prop}
\begin{proof}
Let 
\[
c=b^{-1}a^{-1}ba=\left(\begin{array}{ccc}
1 & 0 & -1\\
0 & 1 & 0\\
0 & 0 & 1
\end{array}\right)
\]
and
\[
\{c^{n}\}_{n\in\mathbb{Z}}=\left\{ \left(\begin{array}{ccc}
1 & 0 & m\\
0 & 1 & 0\\
0 & 0 & 1
\end{array}\right)\left|\begin{array}{c}
\\
m\in\mathbb{Z}\\
\\
\end{array}\right.\right\} 
\]
is the center of $G$. Using the commutation relation $[b,a]=c^{-1}$
it is elementary to show that the set $M_{r}=\{m\,:\, c^{m}\in B_{4r}\}$
contains gaps that grow arbitrarily large as $r\rightarrow\infty$.
Thus we can choose $0\leq s_{r},t_{r}\in M_{r}$ such that $(s_{r},t_{r})\cap M_{r}=\emptyset$
and $t_{r}-s_{r}\rightarrow\infty$. If a sequence $r(i)$ that grows
quickly enough (e.g. if $t_{r(i+1)}-s_{r(i+1)}>r(i)^{2}$), then $\{B_{r(i)}c^{s_{r(i)}-t_{r(i)}}\}$
is an incremental sequence whose elements all contain $1_{G}$. This
proves the claim.
\end{proof}
Theorems \ref{thm:subseq-balls} can be slightly strengthened using
the following version of Theorem \ref{thm:B-is-necessary}:
\begin{thm}
\label{thm:B-necessary-for-bounded-too}In Theorem \ref{thm:B-is-necessary},
if in addition $F_{n}=B_{k(n)}$ are balls in a group, then in the
conclusion we may also assume that $\varphi,\psi\in L^{\infty}(\mu)$.\end{thm}
\begin{proof}
In the $n$-th stage of the construction, instead of setting $\varphi=v$
only on the intervals associated to $h\in H$, choose an appropriate
$m$ and set $\varphi=v/|B_{m/3}|$ on the intervals associated to
$h\in B_{m}$. This $m$ is chosen before $N$ and the $\gamma_{i}$,
and we can ensure that $G_{n}\gamma_{i}\cap B_{m}=\emptyset$ as in
Lemma \ref{cor:strongly-non-B-almost-central-consequences} by simply
choosing an a-priori larger $N$ and discarding some of its elements,
so this does not interfere with the construction. By choosing $N$
large relative to $m$, the $L^{1}$-norm of $\varphi$ still increases
arbitrarily little. Now, since $H\cap F_{n(i)}g\gamma_{i}\neq\emptyset$
for $g\in G_{n}$ and we can assume that $m$ is large enough that
$H\subseteq B_{m/3}$, there is a ball $B_{m/3}\gamma'_{i}\subseteq B_{m}\cap F_{k_{i}}\gamma_{i}$
for some $\gamma'_{i}$ (take $\gamma'_{i}$ to be the point on the
midpoint geodesic from $\gamma_{i}$ to some element of $H\cap F_{n(i)}g\gamma_{j}$).
The proof now carries through.
\end{proof}
Finally, for finitely generated non-abelian free groups it is elementary
that no sequence of balls in the standard generator set is Besicovitch.
We omit the proof. We do not know if this is true for every generating
set, though it seems very likely that it is.

\subsection{\label{sec:Reductions}Some general reductions}

We give here some simple reductions that will be used later. Write
$L_{+}^{1}(\mu)=\{f\in L^{1}(\mu)\,:\, f\geq0\}$. 
\begin{lem}
\label{lem:reduction-to-non-negative}Let $F_{n}\subseteq G$. The
ratio ergodic theorem holds along $\{F_{n}\}$ if and only if for
every action of $G$ there is a $0\neq\psi\in L_{+}^{1}$, such that
$R_{F_{n}}(\varphi,\psi)\rightarrow\int\varphi/\int\psi$ for all
$\varphi\in L_{+}^{1}$.\end{lem}
\begin{proof}
One direction is obvious. For the other fix an action and suppose
that there is a $\psi$ as above. Convergence of $R_{F_{n}}(\varphi,\psi)$
for $\varphi\in L_{+}^{1}$ implies it for all $\varphi\in L^{1}$,
since the operators $R_{F_{n}}(\cdot,\psi)$ are linear and one can
break an arbitrary $L^{1}$ function into a difference of non-negative
ones. Now for any $\varphi,\theta\in L^{1}$ with $\int\theta\neq0$,
the conclusion follows by passing to the limit in the identity $R_{F_{n}}(\varphi,\theta)=R_{F_{n}}(\varphi,\psi)/R_{F_{n}}(\theta,\psi)$. \end{proof}
\begin{lem}
\label{lem:reduction-to-sets-with-identity}There exists a sequence
$F_{n}\subseteq G$ along which the ratio ergodic theorem holds if
and only if there exists such a sequence with, in addition, $1_{G}\in F_{n}$
for all $n$.\end{lem}
\begin{proof}
Only the {}``only if'' direction must be proved. Suppose the ratio
ergodic theorem holds along $\{F_{n}\}$. First suppose some element
$g\in G$ belongs to infinitely many $F_{n}$. Let  $F_{n(i)}$ be
the infinite subsequence of sets containing $g$ and write $F'_{i}=F_{n(i)}g^{-1}$,
so that $1_{G}\in F'_{i}$, and ratio ergodic theorem holds along
$\{F'_{i}\}$ because of the identity $R_{F'_{i}}(\varphi,\psi)=R_{F_{n(i)}}(\varphi,\psi)\circ T^{-g}$.

It remains to deal with the case that no $g$ is in infinitely many
$F_{n}$. In this case, by passing to a subsequence, we can assume
that the sets $F_{n}$ are pairwise disjoint. We shall show that the
ratio ergodic theorem holds along $E_{n}=\bigcup_{i\leq n}F_{i}$.
Since any $g\in E_{1}$ belongs to all of the $E_{n}$, this brings
us back tot he first case that was already established.

Thus, consider an action of $G$ on a non-atomic measure space $(X,\mathcal{B},\mu)$
and $\varphi,\psi\in L_{+}^{1}(\mu)$. Since $\{F_{n}\}$ are pairwise
disjoint, $S_{E_{k}}(\psi)=\sum_{k=1}^{n}S_{F_{k}}(\psi)$, hence
\[
R_{E_{n}}(\varphi,\psi)=\sum_{k=1}^{n}\frac{S_{F_{k}}(\psi)}{S_{E_{n}}(\psi)}R_{F_{k}}(\varphi,\psi)
\]
Since $R_{F_{k}}(\varphi,\psi)\rightarrow\int\varphi/\int\psi$, we
will be done if we show that $S_{E_{n}}(\psi)\rightarrow\infty$.
To see this choose $A\subseteq X$ and $\varepsilon>0$ with $\mu(A)>0$
and $\psi\geq\varepsilon1_{A}$. Choose another set $B\subseteq X$
with $\mu(B)/\mu(A)$ irrational. Then $R_{F_{n}}(1_{B},1_{A})$ are
rational and converge a.e. to the irrational number $\mu(B)\mu(A)$,
so their denominators, which are $S_{F_{n}}(1_{A})$, a.s. tend to
$\infty$ with $k$. Hence $S_{F_{n}}(\psi)\geq\varepsilon S_{F_{n}}(1_{A})\rightarrow\infty$,
concluding the proof. 

We note that if $\bigcup F_{n}=G$ then $S_{E_{n}}(\psi)\rightarrow\infty$
follows directly from conservativity.
\end{proof}
Note that if we only assume the ratios to converge a.e. but not necessarily
to to the limit $\int\psi/\int\varphi$, then the argument above still
works assuming that $S_{E_{n}}(\psi)\rightarrow\infty$ for $0\neq\psi\in L_{+}^{1}$.
This is the case if $\bigcup F_{n}=G$, for example, because ergodicity
on a non-atomic space is the same as conservativity.

We say that $\{F_{n}\}$ is generating if $\bigcup F_{n}$ generates
$G$ as a group. 
\begin{lem}
\label{lem:reduction-to-generating-sequences}If $\{F_{n}\}$ does
not generate then the ratio ergodic theorem fails along $\{F_{n}\}$.\end{lem}
\begin{proof}
Suppose $F_{n}$ lie in a proper subgroup $H<G$ and consider an ergodic
action of $G$ whose restriction to $H$ is non-ergodic (e.g. a product
measures on $\{0,1\}^{G}$ with the shift action). Let $\mathcal{I}$
be the $\sigma$-algebra of $H$-invariant sets and choose functions
$\varphi,\psi\in L^{1}$ that are constant on the atoms of $\mathcal{I}$
but $\mathbb{E}(\varphi|\mathcal{I})/\mathbb{E}(\psi|\mathcal{I})$
is not constant. Clearly $R_{F_{n}}(\varphi,\psi)=\mathbb{E}(\varphi|\mathcal{I})/\mathbb{E}(\psi|\mathcal{I})$
for all $n$, so $R_{F_{n}}(\varphi,\psi)\not\rightarrow\int\varphi/\int\psi$.
\end{proof}

\subsection{\label{sub:Z-infinity}The group $\mathbb{Z}^{\infty}$}

We now turn to $G=\mathbb{Z}^{\infty}$ and the proof of Theorem \ref{thm:Z-infinity},
switching to additive notation. The main ingredient is Proposition
2.7 from \cite{Hochman2006a}. In that paper the Besicovitch property
is called incompressiblity (\cite[Definition 1.9]{Hochman2006a}),
the two notions are the same by Proposition \ref{prop:high-multiplicity-incremental-sequences}.
\begin{prop}
\label{prop:no-B-sequences-in-Z-infinity}If $F_{n}\subseteq\mathbb{Z}^{\infty}$
are finite sets and $\{F_{n}\}$ generates, then $\{F_{n}\}$ is not
Besicovitch.
\end{prop}
Combined with the fact that any sequence in an abelian group is almost
central, the proposition above and Theorem \ref{thm:B-is-necessary}
immediately implies that the ratio ergodic theorem fails along every
generating symmetric sequence $F_{n}\subseteq\mathbb{Z}^{\infty}$
with $0_{G}\in F_{n}$. We now show that the symmetry assumption is
not necessary:
\begin{prop}
\label{prop:no-ratio-for-Z-infinity}The ratio ergodic theorem fails
along any generating sequence $F_{n}\subseteq\mathbb{Z}^{\infty}$
with $0_{G}\in F_{n}$.\end{prop}
\begin{proof}
Suppose that $\{F_{n}\}$ satisfies the ratio ergodic theorem and
$1_{G}\in F_{n}$. Then the same is true for $\{-F_{n}\}$. To see
this, given an action $\{T^{g}\}_{g\in\mathbb{Z}^{\infty}}$ define
an action $\widetilde{T}^{g}=T^{-g}$ (this is an action because $\mathbb{Z}^{\infty}$
is abelian). Then $\sum_{g\in-F_{n}}\varphi(T^{g}x)=\sum_{g\in F_{n}}\varphi(\widetilde{T}^{g}x)$,
and so the ratios over $-F_{n}$ with respect to $T$ are the same
as the ratios over $F_{n}$ with respect to $T$, and so converge
as required. 

Now let $E_{n}=F_{n}\cup(-F_{n})$, which is a symmetric sequence
with $0_{G}\in E_{n}$. Define a probability measures $\nu_{n}$ on
$E_{n}$ by 
\[
\nu_{n}=\frac{1}{2|F_{n}|}\sum_{g\in\pm F_{n}}\delta_{g}
\]
For a finitely supported probability measure $\nu$ on $G$, let 
\[
S_{\nu}(\varphi)=\int\varphi(T^{g}x)\, d\nu(g)
\]
and define $R_{\nu}(\varphi,\psi)=S_{\nu}(\varphi)/S_{\nu}(\psi)$.
Then for any action and $\varphi,\psi$ as in the ratio ergodic theorem,
\[
R_{\nu_{n}}(\varphi,\psi)=\frac{1}{2}(R_{F_{n}}(\varphi,\psi)+R_{-F_{n}}(\varphi,\psi))\rightarrow\int\varphi/\int\psi\qquad\mu\mbox{-a.s.}
\]
While $R_{\nu_{n}}\neq R_{E_{n}}$, on non-negative functions the
two differ by at most a multiplicative constant of $4$, since $\nu_{n}$
is equivalent to the uniform measure $u_{n}$ on $E_{n}$ with Radon-Nikodym
derivative between $1$ and $2$, and $R_{E_{n}}=R_{u_{n}}$. Thus,
if there is an action and functions $\varphi,\psi\in L_{+}^{1}$,
$\int\psi\neq0$, such that $R_{E_{n}}(\varphi,\psi)$ fluctuates
wildly enough (e.g. $\limsup/\liminf>4$), then we have a contradiction
to the convergence of $R_{\nu_{n}}$. Since $\{E_{n}\}$ is symmetric,
contains $1_{G}$, and is almost central and strongly non-Besicovitch
(see discussion preceding this proposition), such an action and pair
of functions can be constructed using exactly the same scheme as in
the proof of Theorem \ref{thm:B-is-necessary}. We omit the details.
\end{proof}
Now, we have already seen that if there is some sequence along which
the ratio ergodic theorem holds then there is also such a sequence
that contains $0_{G}$ (Lemma \ref{lem:reduction-to-sets-with-identity})
and generates (\ref{lem:reduction-to-generating-sequences}). With
these facts in hand, the proposition above proves Theorem \ref{thm:Z-infinity}.

\section{\label{sec:groups-of-polynomial-growth}Groups of polynomial growth}

In this section we prove the ratio ergodic theorem {}``in density''
for groups of polynomial growth (Theorem \ref{thm:polynomial-density-thm}).
After defining the sequence $F_{n}\subseteq G$ in Section \ref{sub:Setup},
the proof follows the standard two-step scheme: in Section \ref{sub:Chacon-Ornstein}
we prove, for fixed $\psi$, that $R_{F_{n}}(\varphi,\psi)$ converges
to the proper limit on a dense family of functions $\varphi\in L^{1}$
(a {}``Chacon-Ornstein lemma''), and in Section \ref{sub:maximal-inequality}
we extend to all $\varphi\in L^{1}$ using a suitable maximal inequality.
Both parts use growth properties of $G$ in an essential way.

\subsection{\label{sub:Setup}The averaging sequence}

Let $G$ be a group of polynomial growth and $B_{n}=A^{n}$ the balls
with respect to some symmetric generating set $A$. By Gromov's theorem
\cite{Gromov1981}, $G$ is virtually nilpotent, and a theorem of
Bass \cite{Bass1972} implies that there are constants $c_{1},c_{2},c$
(moreover, with $c\in\mathbb{N}$) such that 
\[
c_{1}n^{c}\leq|B_{n}|\leq c_{2}n^{c}
\]
Define the $k$-boundary of $B_{n}$ to be 
\[
\partial_{k}B_{n}=B_{n+k}\setminus B_{n-k}
\]
We remark that it is easy to show that $\{B_{n}\}$ is a F\o{}lner
sequence, but we will not use this fact. 

We now define the subsequence $F_{i}=B_{n(i)}$ for which we will
prove Theorem \ref{thm:polynomial-density-thm} (the construction
below can be perturbed in many ways to get a large class other such
sequences). Let
\[
J_{m}=[2^{m-1},2^{m})\cap\mathbb{Z}
\]
We define the index sequence $n(i)$ for $i\in J_{m}$, recursively
in $m=1,2,3\ldots$. For $m=1$ set $n(1)=1$. Now assume we have
defined $n_{i}$ for $i\in\bigcup_{k<m}J_{k}$. Let 
\[
N(m)=n(2^{m-1}-1)
\]
which is the largest value of $n(i)$ defined so far, and set $\{n(i)\}_{i\in J_{m}}$
to be the arithmetic sequence with $|J_{m}|$ terms and gap $3N(m)$,
starting at 
\[
L(m)=|J_{m}|\cdot3N(m)
\]
Thus $n(2^{m-1}+i)=L(m)+i\cdot3N(m)$ for $0\leq i<2^{m}$. 

Note that $\{n(i)\}_{i\in J_{m}}\subseteq[L(m),2L(m))$, hence $N(m)\leq2L(m-1)$,
and, since by the last equation  $L(m-1)=2^{m-1}\cdot3N(m-1)$, we
deduce that
\[
N(m)\leq(62^{m-1})^{m}
\]

Having defined $F_{i}=B_{n(i)}$, for $i\in J_{m}$, set 
\[
F_{i}^{+}=B_{n(i)+N(m)}
\]
and
\[
\partial^{*}F_{i}=\partial_{N(m)}F_{i}
\]
Notice that $F_{i-1}^{+}\cup\partial^{*}F_{i}\subseteq F_{i}^{+}$
and $\partial^{*}F_{i}\cap F_{i-1}^{+}=\emptyset$.

\subsection{\label{sub:Chacon-Ornstein}Convergence on a dense subset of $L^{1}(\mu)$}

For the rest of the section, fix an ergodic measure-preserving action
of $G$ on a $\sigma$-finite measure space $(X,\mathcal{B},\mu)$.
Given $\varphi:X\rightarrow[0,\infty)$ write 
\begin{equation}
\varphi_{i}(x)=\frac{\sum_{g\in\partial^{*}F_{i}}\varphi(T^{g}x)}{\sum_{g\in F_{i-1}^{+}}\varphi(T^{g}x)}\label{eq:boundary-function}
\end{equation}

\begin{lem}
\label{lem:growth-from-large-boundaries}Let $\varphi$ be as above
and $x\in X$. Given $\varepsilon,\delta>0$ suppose that $N\in J_{m}$
and $U\subseteq\{1,\ldots,N\}$ are such that $|U|/N\geq\delta$ and
$\varphi_{i}(x)>\varepsilon$ for $i\in U$. Then, assuming $m$ is
large enough in a manner depending only on $\varepsilon,\delta$,
\[
S_{F_{N}^{+}}(\varphi,x)\geq|F_{N}^{+}|^{2}\varphi(x)
\]
\end{lem}
\begin{proof}
We suppress $x$ in our notation. For $i\in U$ we have by definition
that $S_{\partial^{*}F_{i}}(\varphi)\geq\varepsilon S_{F_{i-1}^{+}}(\varphi)$
and hence $S_{F_{i}^{+}}(\varphi)\geq(1+\varepsilon)S_{F_{i-1}^{+}}(\varphi)$.
Since $\varphi\geq0$, for any $j<i$ we have $S_{F_{i}^{+}}(\varphi)\geq S_{F_{j}^{+}}(\varphi)$.
Starting from $S_{F_{N}^{+}}(\varphi)$ and applying this recursively
to the elements of $U$ in reverse order, we have 
\begin{eqnarray*}
S_{F_{N}^{+}}(\varphi) & \geq & (1+\varepsilon)^{|U|-1}S_{F_{1}^{+}}(\varphi)\\
 & \geq & (1+\varepsilon)^{\delta N-1}\varphi\\
 & \geq & (1+\varepsilon)^{\delta2^{m-1}}\varphi
\end{eqnarray*}
It remains to notice that 
\[
|F_{N}^{+}|\leq|B_{2n(N)}|\leq Cn(N)^{c}\leq CN(m+1)^{c}\leq C(2^{m})^{cm}=C2^{cm^{2}}
\]
for some constant $C$ depending only on $c_{2}$, and that if $m$
is large in a manner depending on $\varepsilon,\delta$ then $C2^{2cm^{2}}\leq(1+\varepsilon)^{\delta2^{m-1}}$.\end{proof}
\begin{thm}
\label{thm:Chacon-Ornstein}Let $\varphi\in L^{1}(\mu)$ with $\varphi\neq0$
and $\varphi\geq0$. Then $\varphi_{i}\xrightarrow{\overline{d}}0$
$\nu$-a.e., where $d\nu=\varphi d\mu$.\end{thm}
\begin{proof}
Fix $\varepsilon>0$. It suffices to show that
\[
\overline{d}\left(i\,:\,\varphi_{i}(x)>\varepsilon\right)=0\qquad\nu\mbox{-a.e.}
\]
Fix $\delta>0$, which we suppress in our notation, and let
\[
E_{N}=\left\{ x\in X\,\left|\,\begin{array}{c}
\delta<\varphi(x)<\delta^{-1}\mbox{ and }\varphi_{i}(x)>\varepsilon\mbox{ for }\\
\mbox{at least a }\delta\mbox{-fraction of }1\leq i\leq N
\end{array}\right.\right\} 
\]
It is enough to show, for every $\delta>0$, that $\nu$-a.e. $x$
belongs to only finitely many $E_{N}$. 

We establish the last claim. Assume, as we may, that $m$ is large
relative to $\varepsilon,\delta$ as in the previous lemma. Let $N\in J_{m}$.
By invariance of $\mu$, we have
\begin{eqnarray}
|F_{N}|\cdot\nu(E_{N}) & = & \int\sum_{g\in F_{N}}1_{E_{N}}(T^{g}x)\varphi(T^{g}x)\, d\mu(x)\nonumber \\
 & = & \int S_{F_{N}}(\varphi\cdot1_{E_{N}})\, d\mu(x)\label{eq:trans}
\end{eqnarray}
Suppose that $g\in F_{N}$ is such that $T^{g}x\in E_{N}$. By the
previous lemma (applied to $\varphi\circ T^{g}$) and the definition
of $E_{N}$, 
\begin{eqnarray*}
S_{F_{N}^{+}}(\varphi)(T^{g}x) & \geq & |F_{N}|^{2}\cdot\varphi(T^{g}x)\\
 & \geq & |F_{N}|^{2}\cdot\delta
\end{eqnarray*}
Since $hg\in(F_{N}^{+})^{2}$ for every $h\in F_{N}^{+}$, we have
shown that if $N\in J_{m}$ then 
\[
S_{F_{N}}(\varphi\cdot1_{E_{N}})>0\qquad\implies\qquad S_{(F_{N}^{+})^{2}}(\varphi)>\delta|F_{N}|^{2}
\]
By definition of $E_{N}$ we have $\varphi(T^{g}y)<\delta^{-1}$ if
$1_{E_{N}}(T^{g}y)\neq0$. Therefore $S_{F_{N}}(\varphi\cdot1_{E_{N}})\leq|F_{N}|\cdot\delta^{-1}$
so, by the implication above, 
\[
S_{F_{N}}(\varphi\cdot1_{E_{N}})\leq|F_{N}|\cdot\delta^{-1}\cdot1_{\{S_{(F_{N}^{+})^{2}}(\varphi)>\delta|F_{N}|^{2}\}}
\]
Integrating this $d\mu$ and using (\ref{eq:trans}) and Markov's
inequality, 
\begin{eqnarray*}
|F_{N}|\nu(E_{N}) & \leq & |F_{N}|\cdot\delta^{-1}\cdot\mu\left(x\,:\, S_{(F_{N}^{+})^{2}}(\varphi)>\delta|F_{N}|^{2}\right)\\
 & \leq & |F_{N}|\cdot\delta^{-2}\left(|F_{N}|^{-2}\cdot\int S_{(F_{N}^{+})^{2}}(\varphi)\, d\mu\right)\\
 & = & \delta^{-2}\cdot|F_{N}|^{-2}\cdot|(F_{N}^{+})^{2}|\cdot\int\varphi\, d\mu
\end{eqnarray*}
Now by polynomial growth and the fact that $(F_{N}^{+})^{2}\subseteq B_{4n(N)}$
we have 
\[
|(F_{N}^{+})^{2}|\leq C\cdot|F_{N}|
\]
for a constant $C$ depending on $c_{1},c_{2},c$, but not on $m$.
Thus, we have shown
\[
\nu(E_{N})\leq\frac{C\int\varphi\, d\mu}{\delta^{2}|F_{N}|^{2}}\leq\frac{C\int\varphi\, d\mu}{\delta^{2}N^{2}}
\]
using the trivial bound $|F_{N}|\geq N$. This is summable, so by
Borel-Cantelli, $\nu$-a.e. $x$ belongs to finitely many $E_{N}$.
\end{proof}
Recall that a co-boundary is a function of the form $\varphi=\tau-\tau^{g}$
for some $g\in G$. It is said to be an $L^{1}$-co-boundary if $\tau\in L^{1}(\mu)$,
and positive if $\tau\geq0$. As a consequence of the theorem above
we obtain a Chacon-Ornstein type statement. Note that in what follows,
$\sigma$-finiteness ensures that statements about strictly positive
$L^{1}$-functions are not vacuous. 
\begin{cor}
\label{cor:Chacon-Ornstein}Let $\tau\in L^{1}(\mu)$ with $\tau>0$
and $\tau\neq0$. Then for every $g\in G$, 
\[
\dlim_{i\rightarrow\infty}R_{F_{i}}(\tau-\tau^{g},\tau)=0\qquad\mu\mbox{-a.e.}
\]
\end{cor}
\begin{proof}
There is an $i_{0}$ such that $g\in F_{i_{0}}$, and for $i>i_{0}$
we have
\[
|R_{F_{i}}(\tau-\tau^{g},\tau)|\leq\left|\frac{S_{\partial^{*}F_{i}}(\tau)}{S_{F_{i}}(\tau)}\right|\leq\left|\frac{S_{\partial^{*}F_{i}}(\tau)}{S_{F_{i-1}^{+}}(\tau)}\right|=\tau_{i}
\]
where $\tau_{i}$ is defined as in (\ref{eq:boundary-function}),
and we have used $\tau\geq0$. From the theorem we conclude that $R_{F_{i}}(\tau-\tau^{g},\tau)\xrightarrow{\overline{d}}0$
at $\tau d\mu$-a.e. point. Since $\tau>0$ the measures $\mu$ and
$\tau d\mu$ are equivalent, and the corollary follows.
\end{proof}
The next conclusion is standard from the previous one.
\begin{prop}
\label{prop:dense-good-functions}Given $0<\psi\in L^{1}(\mu)$, the
set of $\varphi\in L^{1}(\mu)$ such that $R_{F_{i}}(\varphi,\psi)\xrightarrow{\overline{d}}\int\varphi/\int\psi$
a.e. is dense in $L^{1}$.\end{prop}
\begin{proof}
Let us say that $\tau\in L^{1}$ is $\psi$-dominated if $0<\tau<M\psi$
for some $M=M(\tau)$. We claim, first that the convergence in the
statement holds for $\varphi=\tau-\tau^{g}$ where $g\in G$ and $\tau\in L^{1}$
is $\psi$-dominated; and second, that the joint linear span of $\psi$
and the set of such $\varphi$ is dense in $L^{1}$. The two claims
prove the proposition since the limit in question holds trivially
when $\varphi=\psi$ and the operators $R_{F_{i}}(\cdot,\psi)$ are
linear.

For the first statement, let $\varphi=\tau-\tau^{g}$ with $0<\tau\leq M\psi$.
Then $|R_{F_{n}}(\tau,\psi)|\leq M$, hence by the previous corollary,
\[
R_{F_{i}}(\tau-\tau^{g},\psi)=R_{F_{i}}(\tau-\tau^{g},\tau)\cdot R_{F_{i}}(\tau,\psi)\xrightarrow{\overline{d}}0
\]
For the second statement, observe that since $\psi>0$, the set of
differences of $\psi$-dominated functions is dense in the positive
cone of $L^{1}$. It follows easily that the linear span of the set
of co boundaries $\varphi=\tau-\tau^{g}$ with $\tau$ a $\psi$-dominated
function is dense among all $L^{1}$-co-boundaries (note that in general
a co-boundary splits into the difference of two positive co-boundaries).
We now refer to the standard fact that, for ergodic actions and assuming
$\int\psi$$\neq0$, the linear span of $\psi$ and the $L^{1}$-co-boundaries
is dense subspace of $L^{1}(\mu)$ (see e.g. \cite{Feldman2007}). 
\end{proof}

\subsection{\label{sub:maximal-inequality}A density version of the maximal inequality}

The next step is to prove a maximal-type inequality that will allow
to go from the $\overline{d}$-convergence of $R_{F_{i}}(\varphi,\psi)$
on a dense set of $\varphi\in L^{1}(\mu)$ to all of $L^{1}(\mu)$.
Define the density-limsup by 
\[
\dls_{n\rightarrow\infty}a_{n}=\inf\{t\in\mathbb{R}\,:\,\overline{d}(n\,:\, a_{n}>t)=0\}
\]

\begin{lem}
If $a_{n}:X\rightarrow\mathbb{R}$ are measurable then $a=\dls a_{n}$
is measurable.\end{lem}
\begin{proof}
It suffices to show that $\delta_{t}(x)=\overline{d}(n\,:\, a_{n}(x)>t)$
is measurable for each fixed $t$, and this is obvious since $\delta_{t}(x)=\limsup\frac{1}{N}\sum_{n=1}^{N}1_{\{a_{n}>t\}}(x)$. \end{proof}
\begin{thm}
\label{thm:maximal-inequality}Let $\varphi,\psi\in L^{1}(\mu)$ with
$\varphi,\psi\geq0$ and $\int\psi d\mu\neq0$, and write $d\nu=\psi d\mu$.
Then 
\[
\nu\left(\dls_{n\rightarrow\infty}R_{F_{n}}(\varphi,\psi)>t\right)\leq C\frac{\int\varphi\, d\mu}{t}
\]

\end{thm}
Before giving the proof of the maximal inequality, let us use it to
complete the proof of the ratio ergodic theorem. Fix $0<\psi\in L^{1}(\mu)$;
by Lemma \ref{lem:reduction-to-non-negative} it suffices to prove
$R_{F_{n}}(\varphi,\psi)\rightarrow\int\varphi/\int\psi$ for $\varphi\in L^{1}(\mu)$.
Noting that $R_{F_{n}}(\varphi-c\psi,\psi)=R_{F_{n}}(\varphi,\psi)-c$
and setting $c=\int\varphi/\int\psi$, we may further assume that
$\int\varphi d\mu=0$. Let $\varepsilon>0$ and let $\varphi'\in L^{1}(\mu)$
be such that $R_{F_{n}}(\varphi',\psi)\xrightarrow{\overline{d}}0$
and $\left\Vert \varphi-\varphi'\right\Vert _{1}<\varepsilon$, as
exists by the previous proposition. By Theorem \ref{thm:maximal-inequality},
\[
\nu\left(\dls_{n\rightarrow\infty}R_{F_{n}}(|\varphi-\varphi'|,\psi)>\frac{1}{2}\sqrt{\varepsilon}\right)<C\frac{\left\Vert \varphi-\varphi'\right\Vert _{1}}{\sqrt{\varepsilon}}<C\sqrt{\varepsilon}
\]
and by the triangle inequality and $R_{F_{n}}(\varphi',\psi)\xrightarrow{\overline{d}}0$,
\[
\nu\left(\dls_{n\rightarrow\infty}|R_{F_{n}}(\varphi,\psi)|>\sqrt{\varepsilon}\right)\leq\nu\left(\dls_{n\rightarrow\infty}|R_{F_{n}}(|\varphi-\varphi'|,\psi)|>\sqrt{\varepsilon}\right)<C\sqrt{\varepsilon}
\]
From this we conclude that 
\[
\nu\left(\dls_{n\rightarrow\infty}|R_{F_{n}}(\varphi,\psi)|>0\right)=0
\]
Since $\psi>0$ the measures $\nu$ and $\mu$ are equivalent, so
this is the same as $R_{F_{n}}(\varphi,\psi)\xrightarrow{\overline{d}}0$
$\mu$-a.e., as desired.

Turning to the maximal inequality, we will use the following Besicovitch-type
property:
\begin{lem}
\label{lem:bounded-radius-Besicovitch}There is a constant $C$ such
that for any $k,N$, if $\{B_{r(i)}g_{i}\}_{i=1}^{k}$ is an incremental
sequence such that $r(i)\in[N,2N]$, then its multiplicity is at most
$C$.\end{lem}
\begin{proof}
From the assumption $B_{[r(i)/2]}g_{i}$ are pairwise disjoint sets
of size$\geq c_{1}(N/2)^{c}$. If $h\in\bigcap_{i\in I}B_{r(i)}g_{i}$
for some $I\subseteq\{1,\ldots k\}$ then $B_{[r(i)/2]}g_{i}\subseteq B_{3N}h$
for $i\in I$, and the maximal number of such balls is therefore $|B_{3N}h|/|B_{[N/2]}|$.
Since $|B_{3N}|\leq c_{2}(3N)^{c}$ we have $|I|\leq6^{c}c_{2}/c_{1}$. 
\end{proof}
Next we apply a variant of the Vitali covering argument.
\begin{lem}
\label{lem:growth-from-large-ratios}Let $\alpha>0$, let $\varphi,\psi:X\rightarrow[0,\infty)$
and $x\in X$, let $N$ be given and $E\subseteq F_{N}$. Suppose
that for each $g\in E$ there is an $1\leq i(g)\leq N$ such that
$\varphi_{j(g)}(T^{g}x)<\varepsilon$ and $\psi_{j(g)}(T^{g}x)<\varepsilon$,
where $\varphi_{i},\psi_{i}$ are defined as in (\ref{eq:boundary-function}).
Also suppose that $R_{j(g)}(\varphi,\psi,T^{g}x)>\alpha$. Then
\[
\sum_{g\in E}\psi(T^{g}x)\leq\frac{C}{(1-\varepsilon)\alpha}S_{F_{N}^{2}}(\varphi,x)
\]
\end{lem}
\begin{proof}
Let $E_{m}=\{g\in E\,:\, n(g)\in J_{m}\}$. Let $M$ be the maximal
value of $m$ for which $E_{m}\neq\emptyset$. Define $E'_{m}\subseteq E_{m}$
recursively starting from $m=M$ and working down to $n=1$: assuming
we have defined $E'_{k}$ for $k>m$, define $E'_{m}=\{g_{m,1},\ldots,g_{m,\ell(m)}\}$
to be a maximal sequence satisfying the property in the hypothesis
of the previous lemma with respect to $r(g)=n(i(g))$, and also satisfying
$g_{i}\notin\bigcup_{k>m}\bigcup_{i=1}^{\ell(k)}B_{r(g_{k,i})}g_{k,i}$. 

It is easily seen by induction that 
\[
E_{m}\subseteq\bigcup_{k\geq m}\bigcup_{i=1}^{\ell(k)}F_{j(g_{k,i})}g_{k,i}
\]
For $h\in E_{m}$ let 
\[
F'_{j(h)}h=F_{j(h)}h\setminus\bigcup_{k<m}\bigcup_{i=1}^{\ell(k)}F_{j(g_{k,i})}g_{k,i}
\]
so that $\bigcup_{h\in E}F'_{j(h)}h=\bigcup_{h\in E}F_{j(h)}h$. Given
$g\in G$ and $m$, by the previous lemma $g$ belongs to at most
$C$ of the sets $F_{j(g_{m,i})}g_{m,i}$, therefore and if $m_{0}=m_{0}(g)$
is the least index such that this is true for some $j$ and $g\in F_{j(g_{m,i})}g_{m,i}$
then $g$ belongs to $F'_{j(h)}h$ only for some $h\in E_{m_{0}}$(but
no elements $h\in E_{m'}$ for $m'\neq m_{0}$), and to at most $C$
such sets. It follows that
\[
\sum_{g\in E}\psi(T^{g}x)\leq C\sum_{m}\sum_{i=1}^{\ell(m)}\sum_{g\in F'_{j(g_{m,i})}g_{m,i}}\psi(T^{g}x)
\]
Now by our assumptions about $\varphi_{i}(T^{g}x)$ and $\psi_{i}(T^{g}x)$
for $g\in E$, and the fact that $F_{j(h)}h\setminus F'_{j(h)}h\subseteq\partial^{*}F_{j(h)}h$,
we conclude that for all $m$ and $1\leq i\leq\ell(m)$, 
\begin{eqnarray*}
\sum_{g\in F'_{j(g_{m,i})}g_{m,i}}\psi(T^{g}x) & \leq & \sum_{g\in F_{j(g_{m,i})}g_{m,i}}\psi(T^{g}x)\\
 & \leq & \alpha^{-1}\sum_{g\in F_{j(g_{m,i})}g_{m,i}}\varphi(T^{g}x)\\
 & \leq & \frac{\alpha^{-1}}{1-\varepsilon}\sum_{g\in F'_{j(g_{m,i})}g_{m,i}}\varphi(T^{g}x)
\end{eqnarray*}
Combined with the previous inequality, this gives
\begin{eqnarray*}
S_{F_{N}}(\psi1_{E},x) & \leq & \frac{C\alpha^{-1}}{1-\varepsilon}\sum_{m}\sum_{i=1}^{\ell(m)}\sum_{g\in F'_{j(g_{m,i})}g_{m,i}}\varphi(T^{g}x)\\
 & \leq & \frac{C\alpha^{-1}}{1-\varepsilon}\sum_{g\in F_{N}E}\varphi(T^{g}x)\\
 & \leq & \frac{C\alpha^{-1}}{1-\varepsilon}S_{F_{N}^{2}}(\varphi,x)
\end{eqnarray*}
because $\varphi\geq0$ and $F_{N}E\subseteq F_{N}^{2}$ ; the claim
follows.
\end{proof}

\begin{proof}
[Proof of the maximal inequality (Theorem \ref{thm:maximal-inequality})]
Since it suffices to prove the claim with $\varphi$ replaced by $\varphi+\rho\psi$
for arbitrarily small $\rho$, we can assume that $\varphi>0$. Write
$R=\dls R_{F_{i}}(\varphi,\psi)$. Fix $t>0$ and denote 
\[
S=\{x\,:\, R(x)>t\}
\]
For $\delta>0$ let
\[
S_{\delta}=\{x\,:\,\overline{d}(R_{F_{i}}(\varphi,\psi,x)>t+\delta)>\delta\}
\]
Since $S=\bigcup_{\delta>0}S_{\delta}$ and the union is monotone,
it suffices for us to show that $\nu(S_{\delta})\leq\frac{C}{1-\delta}\int\varphi/t$
for a constant $C$ independent of $\delta$. By Theorem \ref{thm:Chacon-Ornstein},
for $\varphi d\mu$-a.e. $x\in S_{\delta}$ we have $\varphi_{i}(x)\xrightarrow{\overline{d}}0$,
with $\varphi_{i}$ as in (\ref{eq:boundary-function}). Since $\varphi,\psi>0$
the measures $\varphi d\mu$ and $\nu=\psi d\mu$ are equivalent,
so this is also true $\nu$-a.e., hence the set 
\[
S'_{\delta}=\{x\,:\,\overline{d}(R_{F_{i}}(\varphi,\psi,x)>t+\delta\mbox{ and }\varphi_{i}(x)<\delta)>\delta\}
\]
differs from $S_{\delta}$ on a set of $\nu$-measure $0$, and it
suffices to bound $\nu(S'_{\delta})$. Now, since $\nu$ is a finite
measure, there is an $N$ such that 
\[
S'_{\delta,N}=\{x\,:\, R_{F_{i}}(\varphi,\psi,x)>t+\delta\mbox{ and }\varphi_{i}(x)<\delta\mbox{ for some }1\leq i\leq N\}
\]
satisfies
\[
\nu(S'_{\delta,N})>\frac{1}{2}\nu(S'_{\delta})
\]
and so it suffices to bound the measure of $S'_{\delta,N}$. 

This now is a direct application of the transference principle and
the previous lemma. We have
\[
|F_{N}|\nu(S'_{\delta,N})=\int S_{F_{N}}(\psi1_{S'_{\delta,N}})\, d\mu
\]
For $x\in X$ let 
\[
E=E_{x}=\{g\in F_{N}\,:\, T^{g}x\in S'_{\delta,N}\}
\]
and for $g\in E$ define $i(g)=i_{x}(g)\in\{1,\ldots,N\}$ to be an
index such that $R_{F_{j(g)}}(\varphi,\psi,T^{g}x)>t$ and $\varphi_{i}(T^{g}x)<\delta$.
In this notation, 
\[
S_{F_{N}}(\psi1_{S'_{\delta,N}},x)=\sum_{g\in E_{x}}\psi(T^{g}x)
\]
and we may apply Lemma \ref{lem:growth-from-large-ratios} at each
$x$, concluding that
\[
|F_{N}|\nu(S'_{\delta,N})\leq\frac{C}{(1-\delta)t}\int S_{F_{N}^{2}}(\varphi)d\mu=\frac{C}{(1-\delta)}|F_{N}^{2}|\int\varphi\: d\mu
\]
The conclusion now follows from the fact that by polynomial growth,
$|F_{N}^{2}|/|F_{N}|=|B_{2n(N)}|/|B_{n(N)}|$ is bounded uniformly
in $N$.
\end{proof}
{

}

\bibliographystyle{plain}
\bibliography{bib}

\noindent{}\emph{\small Current Address: Einstein Institute of Mathematics,
Givat Ram, Hebrew University, Jerusalem 91904, Israel.}{\small \par}

\noindent{}\emph{\small Email: mhochman@math.huji.ac.il}
\end{document}